\documentclass{gen-j-l}%
\usepackage{amsmath}
\usepackage{amssymb}
\usepackage{amsfonts}%
\usepackage{graphicx}
\newtheorem{theorem}{Theorem}[section]
\newtheorem{lemma}[theorem]{Lemma}
\theoremstyle{definition}

\theoremstyle{remark}

\numberwithin{equation}{section}
\theoremstyle{plain}

\newtheorem{corollary}{Corollary}

\copyrightinfo{}{}
\begin{document}
\title[On Langlands-Shahidi for Loop Groups]{On Extending the Langlands-Shahidi Method to Arithmetic Quotients of Loop Groups}
\author{Howard Garland}
\address{Department of Mathematics, Yale University}
\email{garland-howard@yale.edu}
\dedicatory{To Gregg Zuckerman on his 60th birthday}
\begin{abstract}
We discuss certain Eisenstein series on arithmetic quotients of loop groups,
$\hat{G}$, which are associated to cusp forms on finite-dimensional groups
associated with maximal parabolics of $\hat{G}$.

\end{abstract}
\maketitle

\thispagestyle{plain}

\begin{center}
\bigskip

\textsc{Introduction}
\end{center}

\bigskip

In his paper "Euler Products" ([L]), Langlands uses the meromorphic
continuation of Eisenstein series associated with certain cusp forms $\varphi$
on an arithmetic quotient of a Chevalley group $H$, to derive the meromorphic
continuation of certain $L$ -functions associated to $\varphi$ and certain
finite-dimensional representations $\pi$, of the $L$ -group of $H.$ His method
was effective because he had already obtained the meromorphic continuation of
the Eisenstein series ([L2]). However, there were limitations, one being that
in order to apply the Langlands method, $H$ must (up to local isomorphism) be
realized as the semi-simple part of a maximal parabolic subgroup of a
higher-dimensional group, since the method requires using the meromorphic
continuation of an Eisenstein series associated with such a parabolic. There
are then cases which are excluded; e.g., $H_{1}=E_{8} $ and any $\pi,$ and
$H_{2}=SL_{3}(\mathbb{R})\times SL_{3}(\mathbb{R})\times SL_{3}(\mathbb{R})$
and $\pi$ being the representation $\tilde{\pi}, $ the triple tensor product
of the standard representation of $SL_{3}(\mathbb{R}).$ The latter case has
been of particular interest in the theory of automorphic $L$-functions.

Another limitation of the Langlands method has been that, by itself, it does
not yield a holomorphic continuation in cases where that is expected and
desired - only a meromorphic one. This limitation, and other issues (e.g.,
obtaining a functional equation) were dealt with to a significant extent by
Shahidi and others, using the Langlands-Shahidi method.

In the present paper, we describe a possible alternative method for overcoming
these limitations in a number of cases (including the case of ($H_{2}%
,\tilde{\pi}),$ above). This method is based on the theory of Eisenstein
series on arithmetic quotients of loop groups (see e.g., [LG], [R], [AC]), on
a bold suggestion of A. Braverman and D. Kazhdan, which we will describe in
more detail in \S 5 and later in this Introduction, and on a lemma of F.
Shahidi (Lemma 4.1). Then for example, $H_{1},$ $H_{2}$ can each be realized
as the semi-simple part of a maximal parabolic subgroup of $\hat{E}_{8}$ (=
affine $E_{8})$ and $\hat{E}_{6}$ (= affine $E_{6}),$ respectively.
Remarkably, one can then obtain the holomorphic continuation of appropriate
Eisenstein series, with relative ease in a number of cases (again, including
$H_{2}$).

Here we obtain the existence of the desired Eisenstein series (establishing a
Godement criterion (3.11)), and then obtain the Maass-Selberg relations (see
(4.3)). In fact, obtaining the latter proved remarkably easy, thanks to a
result of F. Shahidi (Lemma 4.1). The proof of the Godement criterion (Cor. 1
to Theorem 3.2) depends on the convergence theorem in [AC] for minimal
parabolics, and an argument in [GMRV] for extending convergence theorems for
minimal parabolics to more general ones. It might seem then, that we are in a
good position to extend the results of [L]. However, the same result of
Shahidi, that simplifies the derivation of the Maass-Selberg relations for our
Eisenstein series, also seems at first to prevent the extraction of the
desired $L$-functions from the constant terms of such Eisenstein series: the
problem is that when Shahidi's result does yield something like (4.3), the
reason is precisely that the constant term is "elementary" and does not
involve $L$-functions. This is what happens for example, for the pair
$(\hat{E}_{6},H_{2}).$

A. Braverman and D.Kazhdan proposed a way out of this dilemma. Before
discussing their idea, we note two things: \textit{first}, thanks to (4.3) we
obtain a holomorphic continuation of our Eisenstein series in some cases and
\textit{second}, the methods used in \S \S 2-4 of this paper apply equally
well to number fields and to function fields over finite fields (see e.g., [Lo]).

Braverman and Kazhdan proposed that instead of only computing the constant
terms with respect to "upper triangular" parabolics (which are sufficient for
obtaining convergence and the Maass-Selberg relations), that if possible, one
also computes the constant terms with respect to "lower triangular" parabolics
in order to obtain the $L$-functions (see \S 5, for the definition of upper
and lower triangular parabolic subgroups). Of course in the finite-dimensional
case there is no essential difference between lower and upper triangular
parabolics: lower triangular parabolics are conjugate to upper triangular ones.

As in the finite-dimensional case treated in [L], the computation of the
constant terms will depend on local computations, and in particular, on
certain formulae of Gindikin-Karpelevich type for the lower triangular case.
Such formulae have been conjectured in [BFK] in the non-archimedean case, and
proved there for $F((t)),$ $F$ a finite field. A proof for all non-archimedean
fields will be given in [BGKP]. We will discuss this in \S 5. For now let it
suffice to say that the situation for lower triangular parabolics is more
subtle than for upper triangular ones.

A striking feature of the loop case is that one has reproduced a significant
portion of the theory of automorphic forms for finite-dimensional groups, and
now, considering the two theories together, the loop case and the
finite-dimensional case, one might obtain new results about the classical
theory (e.g., holomorphic continuation of automorphic $L$-functions associated
with (finite-dimensional)cusp forms). The situation is reminiscent of the
proof of Bott periodicity, using the space of based loops of compact symmetric
spaces: one obtains cell decompositions for generalized flag manifolds
-grassmannians for example (parameterized by coset spaces of finite Weyl
groups), and also, cell decompositions of based loop spaces of compact
symmetric spaces (parameterized by coset spaces of affine Weyl groups), and
then, comparing the finite-dimensional cases and the loop cases, Bott
periodicity falls out - a deep result about the homotopy of
(finite-dimensional) compact, symmetric spaces. One might say, we are dealing
here with a "Bott principle" for the theory of automorphic forms: comparing
the loop and finite-dimensional theories, one might obtain new results about
(finite-dimensional) automorphic forms and automorphic $L$ -functions.

I am submitting this paper in honor of Gregg Zuckerman's 60th birthday. Over a
period of many years, I have had the pleasure of collaborating with him on
three papers, and of having an infinite number of discussions. These
discussions ranged over a wide spectrum of mathematics and, as many of his
colleagues know from their own experience, discussions with Gregg are
memorable for his clear grasp of deep mathematical ideas and his uncanny
ability to explain them with utter clarity.

\bigskip

\section{The Setting}

We let \ $A$ be an irreducible, $l\times l,$ classical, Cartan matrix, and we
let $\tilde{A}$ be the corresponding affine, Cartan matrix. We let
$\mathfrak{g}=\mathfrak{g}(A),$ $\mathfrak{\hat{g}}=\mathfrak{g}(\tilde{A})$
be the complex, Kac-Moody Lie algebras corresponding to $A,$ $\hat{A},$
respectively. We let $\mathfrak{g}_{\mathbb{Z}}\subseteq\mathfrak{g},$
$\mathfrak{\hat{g}}_{\mathbb{Z}}\subseteq\mathfrak{\hat{g}}$ denote the
Chevalley $\mathbb{Z}$ -forms, with $\mathfrak{\hat{g}}_{\mathbb{Z}}$
constructed from $\mathfrak{g}_{\mathbb{Z}},$ as in [LA]. We let%
\[
\mathfrak{\hat{g}}^{e}=\mathfrak{\hat{g}}\oplus\mathbb{C}D,
\]%
\[
\mathfrak{\hat{g}}_{\mathbb{Z}}^{e}=\mathfrak{\hat{g}}_{\mathbb{Z}}%
\oplus\mathbb{Z}D
\]
denote the extended, affine, Kac-Moody Lie algebra and $\mathbb{Z}$ -form,
respectively ($D$ being the usual, homogeneous degree operator (see e.g.,
[LG], \S 3(after Prop. 3.3))).

We let $\{e_{i},f_{i},h_{i}\}_{i=1,...,l+1}$ be the Kac-Moody generators of
$\mathfrak{\hat{g}},$ ordered so that $\{e_{i},f_{i},h_{i}\}_{i=1,...,l}$
generate $\mathfrak{g},$ which we may regard as a subalgebra of
$\mathfrak{\hat{g}.}$ We let $\mathfrak{h}$ (resp., $\mathfrak{\hat{h}})$ be
the complex, linear span of the $h_{i},$ $i=1,...,l$ $($resp.,
$i=1,....,l+1),$ and set $\mathfrak{\hat{h}}^{e}=\mathfrak{\hat{h}}%
\oplus\mathbb{C}D,$ $\mathfrak{\hat{h}}_{\mathbb{Z}}=\mathbb{Z}-$span of the
$h_{i},$ $i=1,....,l+1,$ $\mathfrak{\hat{h}}_{\mathbb{Z}}^{e}=\mathfrak{\hat
{h}}_{\mathbb{Z}}\oplus\mathbb{Z}D.$ Recall that $\lambda\in(\mathfrak{\hat
{h}})^{\ast}$ is called \textit{dominant integral}, in case%
\[
\lambda(h_{i})\in\mathbb{Z}_{\geq0},\;i=1,....,l+1.
\]
We further adopt the convention that $\lambda(h_{i})$ must be $>0,$for at
least one $i.$

Given $\lambda\in\mathfrak{\hat{h}}^{\ast}$ dominant integral, we let
$V^{\lambda}$ denote the corresponding irreducible highest weight module of
$\mathfrak{\hat{g}},$ and we let $V_{\mathbb{Z}}^{\lambda}\subseteq
V^{\lambda}$ be a Chevalley $\mathbb{Z}$ -form, as constructed in [LA]. For a
commutative ring with unit, we set $V_{R}^{\lambda}=R\otimes_{\mathbb{Z}%
}V_{\mathbb{Z}}^{\lambda}$ (we also set $\mathfrak{\hat{g}}_{R}=R\otimes
_{\mathbb{Z}}\mathfrak{\hat{g}}_{\mathbb{Z}},$ $\mathfrak{g}_{R}%
=R\otimes_{\mathbb{Z}}\mathfrak{g}_{\mathbb{Z}},$ etc.).

For an algebraically closed field $k,$we let $\hat{G}_{k}^{\lambda}$
$(=\hat{G}_{k}^{\lambda}(\hat{A}))$ be the Chevalley group contained in
$Aut(V_{k}^{\lambda}),$ as defined in [LG], Definition (7.21). For an
arbitrary field $k$ with algebraic closure $\bar{k},$ we let
\[
\hat{G}_{k}^{\lambda}(=\hat{G}_{k}^{\lambda}(\hat{A}))
\]
be the subgroup of $\hat{G}_{\bar{k}}^{\lambda}$ defined by%
\begin{equation}
\hat{G}_{k}^{\lambda}=\{g\in\hat{G}_{\bar{k}}^{\lambda}|g(V_{k}^{\lambda
})=V_{k}^{\lambda}\}.\tag{1.1}%
\end{equation}
In general, for $k$ not algebraically closed, $\hat{G}_{k}^{\lambda}$ so
defined, is larger than the corresponding group of [LG], Definition (7.21).

We let%
\begin{equation}
\hat{G}_{\mathbb{Z}}^{\lambda}(=\hat{\Gamma})=\{\gamma\in\hat{G}_{\mathbb{R}%
}^{\lambda}|\gamma(V_{\mathbb{Z}}^{\lambda})=V_{\mathbb{Z}}^{\lambda
}\}.\tag{1.2}%
\end{equation}
We \textit{adopt the notation} of [LG], [R], [AC]. For%
\[
\nu:\hat{h}_{\mathbb{R}}\rightarrow\mathbb{C}\text{ (real, linear)}%
\]
satisfying Godement's criterion%
\begin{equation}
\operatorname{Re}\nu(h_{i})<-2,\text{ }i=1,....,l+1,\tag{1.3}%
\end{equation}
we have from [R], Theorem 5.1, and [AC], Theorem 12.1:

\begin{theorem}
The infinite sum%
\begin{equation}
\sum_{\gamma\in\hat{\Gamma}/\hat{\Gamma}\cap\hat{B}}\Phi_{\nu}(g\exp
(-rD)\gamma)\tag{1.4}%
\end{equation}
converges absolutely. Moreover, the convergence is uniform on sets $\hat
{K}\Omega_{A}\eta(s)\hat{U}_{\mathcal{D}},$where $s=e^{-r},$ $\Omega
_{A}\subseteq\hat{A}$ is compact, $\hat{U}_{\mathcal{D}}\subseteq\hat{U},$ as
in [AC].
\end{theorem}

The notation is as in [AC], but for the sake of completeness, we add a few
words of explanation: As in [LG], $V_{\mathbb{C}}^{\lambda}$ admits a
positive-definite, Hermitian inner product $\{,\},$which is invariant with
respect to a certain "compact form" $\mathfrak{\hat{k}}\subseteq
\mathfrak{\hat{g}}_{\mathbb{C}}$ (as defined in [LA] ($\mathfrak{\hat{k}}$
being $\mathfrak{k}(\tilde{A})$ of [LA], \S 4). The form $\{,\}$ then
restricts to a real, positive-definite inner product on $V_{\mathbb{R}%
}^{\lambda}$ and $\hat{K}\subseteq\hat{G}_{\mathbb{R}}^{\lambda}$ is defined
by%
\[
\hat{K}=\{k\in\hat{G}_{\mathbb{R}}^{\lambda}|\{k\xi,k\eta\}=\{\xi
,\eta\},\text{ }\xi,\eta\in V_{\mathbb{R}}^{\lambda}\}.
\]
We fix a coherently ordered basis (see [LG], beginning of \S 12 for the
definition) $\mathcal{B},$ say, of $V_{\mathbb{Z}}^{\lambda},$ and we let
$\hat{A}\subseteq G_{\mathbb{R}}^{\lambda}$ be the subgroup of all diagonal
(with respect to the basis $\mathcal{B})$elements with positive entries. We
let $\hat{U}\subseteq G_{\mathbb{R}}^{\lambda}$ be the subgroup of all upper
triangular elements with diagonal elements all equal to one (again, with
respect to $\mathcal{B}).$We then have the Iwasawa decomposition%
\begin{equation}
G_{\mathbb{R}}^{\lambda}=\hat{K}\hat{A}\hat{U}\tag{1.5}%
\end{equation}
(with uniqueness of expression)( see [LG], Lemma 16.14).

Now $\mathfrak{\hat{h}}_{\mathbb{R}}$ is the Lie algebra of $\hat{A}$ and
$\nu$ defines a quasi-character%
\[
\nu:\hat{A}\rightarrow\mathbb{C}^{\times},
\]%
\[
a\mapsto a^{\nu},\text{ }a\in\hat{A}.
\]
Given $g\in\hat{G}_{\mathbb{R}}^{\lambda},$ $g$ has a decomposition%
\[
g=k_{g}a_{g}u_{g},
\]
with respect to (1.5). We then set%
\[
\Phi_{\nu}(g)=a_{g}^{\nu}.
\]

\section{Extensions of the Convergence Theorem (Preliminaries).}

For a field $k$ with algebraic closure $\bar{k},$ we let $\hat{B}_{\bar{k}%
}\subseteq G_{\bar{k}}^{\lambda}$ be the upper triangular subgroup (with
respect to the coherently ordered basis $\mathcal{B}$ ), and $\hat{B}_{k}%
=\hat{B}_{\bar{k}}\cap G_{k}^{\lambda}.$ We let $\hat{P}_{k}\supseteq\hat
{B}_{k}$ be a proper, parabolic subgroup of $\hat{G}_{k}$ $(=\hat{G}%
_{k}^{\lambda};$ we drop the superscript $``\lambda$\textquotedblright\ when
there is no ambiguity about which $\lambda$ we mean).

We consider various subgroups of $\hat{P}_{k}.$ We first let $\alpha
_{1},....,\alpha_{l+1}\in(\mathfrak{\hat{h}}^{e})^{\ast},$ the complex dual of
$\mathfrak{\hat{h}}^{e},$ be the simple roots:%
\[
\alpha_{i}(h_{j})=\tilde{A}_{ij},\text{ }i,j=1,....,l+1,
\]
where%
\[
\tilde{A}=(\tilde{A}_{ij})_{i,j=1,....,l+1}.
\]
We let $\Xi=\{s_{i}\}_{i=1,....,l+1}$ denote the corresponding, simple root
reflections (so $s_{i}$ is the root reflection corresponding to $\alpha_{i}).$
Then the $s_{i}$ generate the (affine) Weyl group $\hat{W}$ of $\mathfrak{\hat
{g}}$ (with respect to $\mathfrak{\hat{h}})$ and for $\theta\subseteq\Xi,$ we
let $W_{\theta}\subseteq\hat{W}$ be the subgroup generated by the elements of
$\theta.$ Then every subgroup $\hat{P}_{k}\supseteq\hat{B}_{k}$ is a group of
the form%
\[
\hat{P}_{k}=\hat{P}_{\theta,k}=\hat{B}_{k}W_{\theta}\hat{B}_{k},
\]
and every proper, parabolic subgroup of $\hat{G}_{k}$ is a conjugate of a
$\hat{P}_{\theta,k}$ for some $\theta\varsubsetneq\Xi$ (in fact, we take this
as the definition of \textquotedblleft proper parabolic\textquotedblright).

We let%
\[
\hat{\Delta}\subseteq(\mathfrak{\hat{h}}^{e})^{\ast}%
\]
be the affine roots of $\mathfrak{\hat{h}},$ and we let $\hat{\Delta}%
_{+}\subseteq\hat{\Delta}$ be the positive roots determined by the choice of
simple roots $\alpha_{1},....,a_{l+1}.$ When convenient, we identify $\Xi$
with the set of simple roots. For $\theta\subseteq\Xi$ (considered then as the
set of simple roots), we let $[\theta]\subseteq\hat{\Delta}$ denote the set of
all roots in $\hat{\Delta}$ which are linear combinations of the elements of
$\theta.$ We let $\hat{H}_{k}\subseteq\hat{B}_{k}$ be the diagonal subgroup
(with respect to the coherently ordered basis $\mathcal{B}),$ and for
$\theta\subseteq\Xi,$ we let%
\[
H_{\theta,k}=\{h\in\hat{H}_{k}|h^{\alpha_{i}}=1,\text{ }\alpha_{i}\in\theta\}.
\]
For an algebraically closed field $\bar{k},$ we let $L_{\theta,\bar{k}%
}\subseteq\hat{G}_{\bar{k}}$ be the subgroup generated by elements
$\{\chi_{\alpha}(u)\}_{\alpha\in\lbrack\theta],u\in\bar{k}}.$ For an arbitrary
field $k$ with algebraic closure $\bar{k},$we set%
\[
L_{\theta,k}=L_{\theta,\bar{k}}\cap\hat{G}_{k}.
\]

One lets $\hat{U}_{\theta,k}\subseteq\hat{P}_{\theta,k}$ be the pro-unipotent
radical; then%
\[
\hat{P}_{\theta,k}=M_{\theta,k}\hat{U}_{\theta,k},
\]
with $\hat{U}_{\theta,k}$ normal, $M_{\theta,\bar{k}}=L_{\theta,\bar{k}%
}H_{\theta,\bar{k}},$ and $M_{\theta,k}=M_{\theta,\bar{k}}\cap\hat{G}_{k} $
(see [LG2], Theorem 6.1).

We now consider the case when $k=\mathbb{R}.$ For $k=\mathbb{R},$we set
$\hat{G}^{\lambda}$ $(=\hat{G})=\hat{G}_{\mathbb{R}}^{\lambda},$ and%
\[
\hat{H}=\hat{H}_{\mathbb{R}},
\]%
\[
H_{\theta}=H_{\theta,\mathbb{R}}\text{ ,}%
\]%
\[
\hat{P}_{\theta}=\hat{P}_{\theta,\mathbb{R}}\text{ ,}%
\]%
\[
L_{\theta}=L_{\theta,\mathbb{R}},\text{ }M_{\theta}=M_{\theta,\mathbb{R}},
\]
etc.. We let%
\[
Z\subseteq\hat{H}%
\]
be the subgroup of all elements whose diagonal elements are $\pm1.$ We let%
\[
\lambda:\hat{H}\rightarrow\mathbb{C}^{\times}%
\]
be a quasi-character such that $\lambda|_{Z}$ is identically equal to $1.$ We
may identify $\lambda$ with a real, linear function%
\[
\lambda:\mathfrak{\hat{h}}_{\mathbb{R}}\rightarrow\mathbb{C},
\]
where $\mathfrak{\hat{h}}_{\mathbb{R}}$ is the Lie algebra of $\hat{H}.$

We are assuming $\mathfrak{g}=\mathfrak{g}(A)$ is simple (for recall, we
assumed at the beginning, that $A$ is irreducible), and that we have ordered
the $h_{i}$ so that $h_{1},....,h_{l}$ span $\mathfrak{h},$ the Cartan
subalgebra of $\mathfrak{g}(A)\subseteq\mathfrak{g}(\hat{A}).$ We let
$\alpha_{0}$ be the corresponding highest root of $\mathfrak{g}(A)$
($\alpha_{0}\in\mathfrak{h}^{\ast},$ the complex dual of $\mathfrak{h})$ and
we let%
\[
c=h_{\alpha_{0}}+h_{l+1}\in\mathfrak{\hat{h}},
\]
($h_{\alpha_{0}}$ denoting the coroot corresponding to $\alpha_{0}).$ Then $c$
spans the center of $\mathfrak{\hat{g}}.$ We have the extended Cartan%
\[
\mathfrak{\hat{h}}^{e}=\mathfrak{\hat{h}}\oplus\mathbb{C}D
\]%
\[
=\mathfrak{h}\oplus\mathbb{C}c\oplus\mathbb{C}D,
\]
and a corresponding decomposition of $(\mathfrak{\hat{h}}^{e})^{\ast}$, the
complex dual of $\mathfrak{\hat{h}}^{e},$%
\[
(\mathfrak{\hat{h}}^{e})^{\ast}=\mathfrak{h}^{\ast}\oplus\mathbb{C\lambda
}_{l+1}\oplus\mathbb{C}\iota,
\]
where, e.g.,%
\[
\lambda_{l+1}(c)=1,\text{ }\iota(D)=1.
\]
Note that $\iota$ is the generating, imaginary root and $\lambda_{l+1}$ is the
$l+1^{st}$ fundamental weight defined by%
\[
\lambda_{l+1}(h_{i})=\left\{
\begin{array}
[c]{c}%
0,\text{ }i\neq l+1\\
1,\text{ }i=l+1.
\end{array}
\right.
\]

Now with these conventions and notations, assume%
\[
\theta=\theta_{0}=\{\alpha_{1},....,\alpha_{l}\};
\]
then (over $\mathbb{R}$)%
\[
\mathfrak{h}_{\theta}=\mathbb{R}c\text{ (}\mathfrak{h}_{\theta}=\text{ Lie
algebra of }H_{\theta}),
\]%
\[
L_{\theta}=G,
\]
a real, connected Lie group with Lie algebra $\mathfrak{g}_{\mathbb{R}%
}=\mathbb{R}\otimes_{\mathbb{Z}}\mathfrak{g}_{\mathbb{Z}}.$

\section{Extensions of the Convergence Theorem (Continued).}

We return to $\nu$ and to $\Phi_{\nu},$ as considered in (1.3) and (1.4) - see
the exact definition of $\Phi_{\nu}$ at the end of \S 1. We now further assume
that $\nu$ is $\mathbb{R}$ -valued; i.e., that%
\begin{equation}
\nu:\mathfrak{\hat{h}}_{\mathbb{R}}\rightarrow\mathbb{R}\tag{3.1}%
\end{equation}
(so now $\nu(h_{i})<-2,$ $i=1,....,l+1).$ We let $\hat{P}=\hat{P}_{\theta
}\supseteq\hat{B}$ be a proper, parabolic subgroup, and we consider the sum
(1.4):%
\begin{equation}
\sum_{\gamma\in\hat{\Gamma}/\hat{\Gamma}\cap\hat{B}}\Phi_{\nu}(g\exp
(-rD)\gamma)\tag{3.2}%
\end{equation}%
\[
=\sum_{\gamma\in\hat{\Gamma}/\hat{\Gamma}\cap\hat{P}}\text{ }\sum_{\beta
\in\hat{\Gamma}\cap\hat{P}/\hat{\Gamma}\cap\hat{B}}\Phi_{\nu}(g\exp
(-rD)\gamma\beta).
\]

Thanks to our present assumption (3.1), the series on either side of (3.2) are
in fact series of positive terms, and since the left side is convergent
(Theorem 1.1), so is the right side, and in particular, the series%
\begin{equation}
\sum_{\beta\in\hat{\Gamma}\cap\hat{P}/\hat{\Gamma}\cap\hat{B}}\Phi_{\nu}%
(g\exp(-rD)\gamma\beta)\tag{3.3}%
\end{equation}
($\gamma\in\hat{\Gamma}/\hat{\Gamma}\cap\hat{P}$ now fixed) is convergent (and
of course, is absolutely convergent, since it is a sum of positive terms).

However, the series (3.3) is in fact a convergent Eisenstein series for the
(finite-dimensional) reductive group%
\[
M_{\theta}%
\]
(where recall $\hat{P}=\hat{P}_{\theta}$). More precisely, let%
\[
\pi:\hat{P}_{\theta}\rightarrow M_{\theta}%
\]
be the projection. Let%
\[
K_{\theta}=\hat{K}\cap\hat{P}_{\theta}=\hat{K}\cap M_{\theta},
\]
where $\hat{K}$ is as in \S 1; then of course%
\[
\pi(\hat{K}\cap\hat{P}_{\theta})=K_{\theta}.
\]

Consider the elements $g,$ $\gamma$ appearing in (3.3). We have%
\[
g\exp(-rD)\gamma=k_{g\gamma}m_{g\gamma}\exp(-rD)u_{g\gamma},\text{ }%
k_{g\gamma}\in\hat{K},\text{ }m_{g\gamma}\in M_{\theta},\text{ }u_{g\gamma}%
\in\hat{U}_{\theta};
\]
then $(\beta$ as in (3.3))%
\[
\Phi_{\nu}(g\exp(-rD)\gamma\beta)=\Phi_{\nu}(m_{g\gamma}\exp(-rD)\beta),
\]
and the sum (3.3) becomes%
\begin{equation}
\sum_{\beta\in\hat{\Gamma}\cap\hat{P}/\hat{\Gamma}\cap\hat{B}}\Phi_{\nu
}(m_{g\gamma}\exp(-rD)\beta).\tag{3.3$^\prime$}%
\end{equation}

Set%
\[
B_{\theta}=\pi(\hat{B})\subseteq M_{\theta},
\]%
\[
\Gamma_{\theta}=\pi(\hat{\Gamma}\cap\hat{P}_{\theta}).
\]
Then (3.3$^{\prime})$ equals%
\begin{equation}
\sum_{\beta\in\Gamma_{\theta}/\Gamma_{\theta}\cap B_{\theta}}\Phi_{\nu
}(m_{g\gamma}\exp(-rD)\beta),\tag{3.3$^\prime\prime$}%
\end{equation}
where $\Gamma_{\theta}$ is an arithmetic subgroup of $M_{\theta}.$ For $m\in
M_{\theta},$ we set%
\[
\Phi_{\nu}(m)=\Phi_{\nu}(m\exp(-rD)),
\]
and we let%
\[
\Gamma_{\theta}^{r}=_{df}\exp(-rD)\Gamma_{\theta}\exp(rD);
\]
Noting that $\exp(-rD)$ normalizes $B_{\theta},$ we have that the sum
(3.3$^{\prime\prime})$ equals%
\begin{equation}
\sum_{\beta\in\Gamma_{\theta}^{r}/\Gamma_{\theta}^{r}\cap B_{\theta}}\Phi
_{\nu}(m_{g\gamma}\beta),\tag{3.4}%
\end{equation}
which is an Eisenstein series for the pair $(M_{\theta},\Gamma_{\theta}^{r}).$

On the one hand, the absolute convergence of (3.4) follows from that of (3.2).
On the other hand, our assumption that $\nu$ in (3.1) satisfies Godement's
criterion $(\nu(h_{i})<-2,$ $i=1,....,l+1)$ in fact implies that (3.4) is
absolutely convergent, thanks to Godement's criterion for finite-dimensional groups.

It is useful to give an alternate description of the sum (3.4). Let%
\[
\mathfrak{h}(\theta)\subseteq\mathfrak{\hat{h}}_{\mathbb{R}},
\]
be the (real) linear span of the $h_{i},$ $(s_{i}\in\theta);$ then
$\mathfrak{h}(\theta)$ may be regarded as the Cartan subalgebra of
$\mathfrak{l}_{\theta},$ the Lie algebra of $L_{\theta}.$ We let%
\[
H(\theta)\subseteq L_{\theta}%
\]
denote the group generated by the elements%
\[
h_{\alpha_{i}}(s),\text{ }\alpha_{i}\in\theta,s\in\mathbb{R}^{\times},
\]
so that $\mathfrak{h}(\theta)$ is the Lie algebra of $H(\theta).$

We let%
\[
L_{\theta}^{\prime}=M_{\theta}/Z(M_{\theta}),\text{ }Z(M_{\theta})\text{
=center of }M_{\theta}.
\]
We then have that%
\[
\Phi_{\nu}|_{L_{\theta}}%
\]
is the lift of a function $\Phi_{\nu}^{\prime}$ on $L_{\theta}^{\prime}. $ If
we let%
\[
\tilde{\omega}:M_{\theta}\rightarrow L_{\theta}^{\prime}%
\]
denote the projection, if we let $H(\theta)^{\prime}=\tilde{\omega}%
(H(\theta))$ and $A(\theta)^{\prime}$ denote the identity component of
$H(\theta)^{\prime},$ $K_{\theta}^{\prime}=\tilde{\omega}(K_{\theta})$ and
$U_{\theta}^{\prime}=\tilde{\omega}(\hat{U}\cap M_{\theta}),$ then we have the
Iwasawa decomposition%
\[
L_{\theta}^{\prime}=K_{\theta}^{\prime}A(\theta)^{\prime}U_{\theta}^{\prime},
\]
and
\[
\Phi_{\nu}^{\prime}(k^{\prime}a^{\prime}u^{\prime})=(a^{\prime})^{\nu},\text{
}k^{\prime}\in K_{\theta}^{\prime},\text{ }a^{\prime}\in A(\theta)^{\prime
},\text{ }u^{\prime}\in U_{\theta}^{\prime},
\]
(where we may identify $A(\theta)^{\prime}$ with a subgroup of $L_{\theta},$
in order to define $(a^{\prime})^{\nu})$; then we have for $m\in L_{\theta},$%
\[
\sum_{\beta\in\Gamma_{\theta}^{r}/\Gamma_{\theta}^{r}\cap B_{\theta}}\Phi
_{\nu}(m\beta)
\]%
\begin{equation}
=\sum_{\beta\in(\Gamma_{\theta}^{r})^{\prime}/(\Gamma_{\theta}^{r})^{\prime
}\cap\tilde{\omega}(B_{\theta})}\Phi_{\nu}^{\prime}(\tilde{\omega}%
(m)\beta),\tag{3.5}%
\end{equation}
which is a convergent Eisenstein series on $L_{\theta}^{\prime}$ (with respect
to $(\Gamma_{\theta}^{r})^{\prime}=_{df}\tilde{\omega}(\Gamma_{\theta}^{r})$
and the Borel subgroups $\tilde{\omega}(B_{\theta})\subseteq L_{\theta
}^{\prime}$.) Let%
\[
E_{\theta}(m)=\tilde{E}_{\theta}(\tilde{\omega}(m))
\]
denote the convergent sum (3.5) (which as noted, is an Eisenstein series).

On the other hand, we consider%
\[
F_{\theta}(m)=_{df}\sum_{\beta\in\Gamma_{\theta}^{r}/\Gamma_{\theta}^{r}\cap
B_{\theta}}\Phi_{\nu}(m\beta);
\]
then%
\[
F_{\theta}E_{\theta}^{-1}(\cdot)
\]
is entirely determined by its restriction to $Z(M_{\theta}),$ and indeed%
\[
Z(M_{\theta})\subseteq\hat{H},
\]
and%
\[
F_{\theta}E_{\theta}^{-1}(z)=z^{\nu},\text{ }z\in Z(M_{\theta}).
\]
Some explanation is required here, since, strictly speaking, the
quasicharacter $\nu$ is only defined on the identity component $\hat
{A}\subseteq\hat{H},$ and hence only on the identity component $A(M_{\theta})$
of $Z(M_{\theta}).$ However, one has a homomorphism%
\[
\hat{H}\overset{\sigma}{\rightarrow}\hat{A}%
\]
given by absolute value (each $h\in\hat{H}$ is represented by a diagonal
matrix, and $\sigma(h)$ is just the corresponding matrix of absolute values,
and $\nu$ on $\hat{H}$ is just taken to be $\nu\circ\sigma.$

We set%
\[
\xi_{\nu}(\cdot)=F_{\theta}E_{\theta}^{-1}(\cdot);
\]
then of course,%
\[
F_{\theta}=E_{\theta}\xi_{\nu}%
\]
(where $\xi_{\nu}|_{L_{\theta}}\equiv1).$

But then, since (3.3) equals (3.4), we have%
\begin{equation}
\sum_{\beta\in\hat{\Gamma}\cap\hat{P}/\hat{\Gamma}\cap\hat{B}}\Phi_{\nu}%
(g\exp(-rD)\gamma\beta)=E_{\theta}(m_{g\gamma})\xi_{\nu}(m_{g\gamma
}),\tag{3.6}%
\end{equation}
and so (see (3.2))%
\begin{equation}
\sum_{\gamma\in\hat{\Gamma}/\hat{\Gamma}\cap\hat{B}}\Phi_{\nu}(g\exp
(-rD)\gamma)=\sum_{\gamma\in\hat{\Gamma}/\hat{\Gamma}\cap\hat{P}}E_{\theta
}(m_{g\gamma})\xi_{\nu}(m_{g\gamma}).\tag{3.7}%
\end{equation}

Now $E_{\theta}(\cdot)$ is by definition the lift of an Eisenstein series on
$L_{\theta}^{\prime},$ and hence is bounded below by some $\kappa>0.$ It
follows that%
\begin{equation}
\sum_{\gamma\in\hat{\Gamma}/\hat{\Gamma}\cap\hat{P}}\xi_{\nu}(m_{g\gamma
})<\infty\tag{3.8}%
\end{equation}
since the series on either side of (3.7) is (absolutely) convergent. (this
method of deriving convergence for general parabolics from convergence for
minimal ones, comes from [GMRV]).

It is instructive to analyze the element $m_{g\gamma}\in M_{\theta}.$ Recall
(before (3.3$^{\prime}))$ the equation%
\[
g\exp(-rD)\gamma=k_{g\gamma}m_{g\gamma}\exp(-rD)u_{g\gamma}.
\]
$M_{\theta}$ is then a direct product%
\[
M_{\theta}=\tilde{L}_{\theta}A(M_{\theta}),
\]
where recall that $A(M_{\theta})$ is the identity component of $Z(M_{\theta
}),$ and where $\tilde{L}_{\theta}\ $is a subgroup or $\hat{G}^{\lambda}$
containing $L_{\theta}$ with $\tilde{L}_{\theta}/L_{\theta}$ finite. Of course
$m\in M_{\theta}$ then has a corresponding expression $m=la,$ $l\in\tilde
{L}_{\theta},$ $a\in A(M_{\theta}),$ and in particular, this is the case for
$m_{g\gamma}$. Moreover%
\[
\xi_{\nu}(la)=a^{\nu},\text{ }l\in\tilde{L}_{\theta},a\in A(M_{\theta}).
\]

Given $g\in\hat{G}^{\lambda},$ and the decomposition%
\[
\hat{G}^{\lambda}=\hat{K}\hat{P}%
\]%
\[
=\hat{K}\tilde{L}_{\theta}A(M_{\theta})\hat{U}_{\theta},
\]
we have that $g\in\hat{G}^{\lambda}$ has a corresponding decomposition%
\[
g=k_{g}l_{g}a_{g}u_{g},
\]%
\[
k_{g}\in\hat{K},l_{g}\in\tilde{L}_{\theta},a_{g}\in A(M_{\theta}),u_{g}\in
\hat{U}_{\theta}.
\]
One has:

\begin{lemma}
$a_{g}$ is uniquely determined by $g.$
\end{lemma}

We briefly sketch the proof, which is a straightforward application of
representation theory. Let $\mu$ be the sum of those fundamental weights
$\lambda_{i}$ such that $\alpha_{i}$ is not in $\theta.$ Then the group
$\hat{G}^{\lambda}$ acts on $V^{m\mu}$ for some positive multiple $m\mu$ of
$\mu.$(see [LG], Prop. 20.2). We have a positive-definite, Hermitian inner
product $\{,\},$ on $V^{m\mu},$ as in \S 1, and we let $v_{m\mu}$ be a highest
weight vector of norm one, with respect to $||\cdot||,$ the norm corresponding
to $\{,\}.$ Then $||g\cdot v_{m\mu}||=a_{g}^{m\mu},$ and the lemma follows.

We now specialize to the case when%
\[
\theta=\theta_{i_{0}}=\Xi-\{\alpha_{i_{0}}\},
\]
for a single, simple root $\alpha_{i_{0}}.$ We have a relation
\[
\sum_{i=1}^{l}n_{i}\alpha_{i}^{\nu}+h_{l+1}=c,
\]
($c$ defined as in \S 2) where if $\alpha_{0}$ is the highest root of $g(A),
$then%
\[
\alpha_{0}^{\nu}=\sum_{i=1}^{l}n_{i}\alpha_{i}^{\nu},
\]
with $\alpha_{0}^{\nu},$ $\alpha_{i}^{\nu}$ denoting the coroots corresponding
to $\alpha_{0},$ $\alpha_{i},$ respectively.

Now if\qquad%
\begin{equation}
\nu(\alpha_{i}^{\nu})<-2,\text{ }i=1,....,l+1,\tag{3.9}%
\end{equation}
then of course%
\[
\nu(c)<-2(1+\sum_{i=1}^{l}n_{i});
\]
we set%
\[
g=1+\sum_{i=1}^{l}n_{i}%
\]
(which we will call the dual Coxeter number). Then%
\begin{equation}
\nu(c)<-2g.\tag{3.10}%
\end{equation}

On the other hand, if%
\[
\tilde{\nu}:\mathbb{R}c\rightarrow\mathbb{R}%
\]
satisfies (3.10), then $\tilde{\nu}$ is the restriction of some%
\[
\nu:\mathfrak{\hat{h}}_{\mathbb{R}}\rightarrow\mathbb{R}%
\]
satisfying (3.9).

If $g\in\hat{G}^{\lambda},$then%
\[
g\exp(-rD)\gamma=k_{g\gamma}l_{g\gamma}a_{g\gamma}\exp(-rD)u_{g\gamma},\text{
}k_{g\gamma}\in\hat{K},l_{g\gamma}\in\tilde{L}_{\theta},a_{g\gamma}\in
A(M_{\theta}),u_{g\gamma}\in\hat{U}_{\theta},
\]
and we set%
\[
\xi_{\nu}(g\gamma)=a_{g\gamma}^{\nu};
\]
then our above argument shows

\begin{theorem}
For $\theta=\theta_{i_{0}},$ as above and for%
\[
\nu:\mathfrak{\hat{h}}_{\mathbb{R}}\rightarrow\mathbb{R}%
\]
a real, linear function such that%
\[
\nu(c)<-2g,
\]
we have%
\[
\sum_{\gamma\in\hat{\Gamma}/\hat{\Gamma}\cap\hat{P}}a_{g\gamma}^{\nu}<\infty.
\]

\end{theorem}

\begin{corollary}
If $\tilde{\varphi}$ is a cusp form on $L_{\theta_{i_{0}}}^{\prime}$ which is
rapidly decreasing (e.g., an eigenfunction for the center of the universal
enveloping algebra of $L_{\theta_{i_{0}}}^{\prime}),$ if $\varphi
=\tilde{\varphi}\circ\tilde{\omega},$ and if
\[
\nu:\mathfrak{\hat{h}}_{\mathbb{R}}\rightarrow\mathbb{C}%
\]
is a real linear function, such that%
\begin{equation}
\operatorname{Re}(\nu)(c)<-2g,\tag{3.11}%
\end{equation}
then%
\begin{equation}
\sum_{\gamma\in\hat{\Gamma}/\hat{\Gamma}\cap\hat{P}}\varphi(m_{g\gamma}%
)\xi_{\nu}(a_{g\gamma}),\text{ }g\in\hat{G},\tag{3.12}%
\end{equation}
converges absolutely.
\end{corollary}

\begin{proof}
$\varphi$ is of course bounded. If $\nu$ satisfies (3.11), we can dominate the
series (3.12) by%
\[
\sum_{\gamma\in\hat{\Gamma}/\hat{\Gamma}\cap\hat{P}}\xi_{\operatorname{Re}\nu
}(a_{g\gamma}),
\]
which converges by Theorem 3.2.
\end{proof}

\section{Shahidi's Argument}

Thanks to the Corollary to Theorem 3.2, we have Eisenstein series on loop
groups which are associated to certain cusp forms on finite-dimensional,
semi-simple groups. For example, consider the affine Dynkin diagram associated
with $E_{6},$ with the vertices numbered as in [Bourb] (assign the number 7 to
the vertex corresponding to the negative of the highest root). Consider (with
this numbering)%
\[
\theta_{4}=\{\alpha_{1},\alpha_{2},\alpha_{3},\alpha_{5},\alpha_{6},\alpha
_{7}\},
\]
and the subgroups%
\[
M_{\theta_{4}},\text{ }L_{\theta_{4}}%
\]
of $\hat{G}^{\lambda}.$ Now take%
\[
\lambda=\lambda_{4},
\]
the fundamental weight corresponding to node 4 (in the numbering of [Bourb]).
Then $L_{\theta_{4}}$ locally isomorphic to
\[
SL_{3}(\mathbb{R})\times SL_{3}(\mathbb{R})\times SL_{3}(\mathbb{R}).
\]
Hence, starting with a cusp form $\varphi$ on $L_{\theta_{4}}^{\prime},$ one
obtains an Eisenstein series on $\hat{G}^{\lambda_{4}}$ (denote this group by
$\hat{E}_{6}).$

But then, motivated by [L], one can ask to find the constant term for such an
Eisenstein series, and then hope to obtain (for suitable $\varphi)$ an
expression involving $L$-functions associated with $\varphi$, and certain
representations of the Langlands dual $\hat{E}_{6}^{L}.$

However, at first, this strategy seemed doomed to fail: the constant terms
with respect to parabolics $\hat{P}_{\theta}$, as above, do not yield
$L$-functions as in [L]. The problem is that $M_{\theta_{4}}$ is not
self-associate, this being an instance of Shahidi's lemma (see [S]):

\begin{lemma}
Let $\theta=\theta_{i},$ $i=1,....,l+1;$ then $M_{\theta}$ is not self-associate.
\end{lemma}

As the proof in [S] is not terribly long, we include it here for the
convenience of the reader: We have set $\mathfrak{h}(\theta)$ equal to the
real linear span of the $h_{j},$ $j\neq i$ $(\theta=\theta_{i}),$ and so
$\mathfrak{h}(\theta)$ is the Lie algebra of $H(\theta)\subseteq L_{\theta},$
the subgroup generated by the $h_{\alpha_{j}}(s),$ $j\neq i,$ $s\in
\mathbb{R}^{\times}.$ We let $\hat{W}_{\theta}$ denote the subgroup of the
Weyl group generated by the $s_{j},$ $j\neq i,$ and we set%
\[
w_{0}^{\theta}=\text{ longest element in }\hat{W}_{\theta}.
\]
Assume then, that there is an element $w_{0}\in\hat{W}$ such that%
\[
w_{0}(\theta)=\theta,\text{ }w_{0}(\alpha_{i})<0
\]
(this being the definition of $M_{\theta}$ being self-associate). Then%
\[
w_{0}w_{0}^{\theta}(\theta)=-\theta,
\]
while%
\[
w_{0}w_{0}^{\theta}(\alpha_{i})<0.
\]

To see this last assertion, we note that $w_{0}^{\theta}(\alpha_{i})$ has an
expression%
\[
w_{0}^{\theta}(\alpha_{i})=\alpha_{i}+\sum_{j\neq i}k_{j}\alpha_{j},
\]
and then%
\begin{equation}
w_{0}w_{0}^{\theta}(\alpha_{i})=w_{0}(\alpha_{i})+\sum_{j\neq i}k_{j}^{\prime
}\alpha_{j},\tag{4.1}%
\end{equation}
(since $w_{0}(\theta)=\theta).$ But, by assumption, $w_{0}(\alpha_{i})$ is
negative, and%
\[
w_{0}(\alpha_{i})=\sum_{j=1}^{l+1}b_{j}\alpha_{j},\text{ with }b_{i}\neq0
\]
(otherwise $w_{0}(\alpha_{j})\in\lbrack\theta],$ the roots which are linear
combinations of the elements of $\theta,$ for all $j,$ and this is not
possible). Hence%
\[
w_{0}w_{0}^{\theta}(\alpha_{i})<0,\text{ by (4.1).}%
\]
Hence $w_{0}w_{0}^{\theta}(\Delta_{+})=\Delta_{-},$ and in particular,
$w_{0}w_{0}^{\theta}$ maps positive imaginary roots to negative roots. This is
not possible, and so we obtain Lemma 4.1.

Now Lemma 4.1 seems to have an unfortunate consequence: At least for certain
maximal parabolic subgroups, one can not obtain non-trivial constant terms
from Eisenstein series associated with cusp forms of the reductive part.

But as Shahidi noted, there is also good news here: the theory of Eisenstein
series associated to cusp forms for the reductive part of a maximal parabolic
subgroup of a loop group does not depend on the knowledge of any new
$L$-functions, and so might be more accesible than otherwise. In fact, the
constant term of such Eisenstein series can be extremely simple. In the
notation of the Corollary to Theorem 3.2 , if $E_{\varphi}(\nu)$ denotes the
convergent sum (3.12), and if we consider the case of $\hat{E}_{6}$ and
$\theta_{4}$ (as described above) then $\hat{P}_{\theta_{4}}$ is not associate
to \textit{any} other parabolic and is not self associate by Lemma 4.1. We
then have as a consequence of Lemma 4.1:

\begin{lemma}
The constant term%
\[
E_{\varphi}(\nu)_{\hat{U}_{\theta_{i}}}(g)=_{df}\int_{\hat{U}_{\theta_{i}%
}/\hat{U}_{\theta_{i}}\cap\hat{\Gamma}}E_{\varphi}(\nu)(g\exp(-rD)u)du
\]
is equal to $\varphi(m_{g})\xi_{\nu}(a_{g}).$
\end{lemma}

In our example of $(\hat{E}_{6},\theta_{4})$ then, this is the only non-zero
contribution to the constant term. In [MS2], we extended Arthur's definition
of truncation to loop groups ([MS2], Definition 3.2). In the notation of that
paper,  we have as a consequence of the simplicity of the constant terms we
have just discussed, that for $(\hat{E}_{6},\theta_{4}):$%
\[
\wedge^{H_{0}}E_{\varphi}(\nu)(g\eta(s))\text{ (}s=\exp(-r),\text{ }%
\eta(s)=\exp(-rD),\text{ as in [MS 2])}%
\]%
\[
=\sum_{\gamma\in\hat{\Gamma}/\hat{\Gamma}\cap\hat{P}_{\theta_{4}}}(1-\hat
{T}_{\theta_{4},H_{0}}(g\eta(s)\gamma))\varphi(m_{g\gamma})\xi_{\nu
}(a_{g\gamma}),
\]
where $m_{g\gamma},$ $a_{g\gamma}$ are recall, defined by%
\[
g\eta(s)\gamma=k_{g\gamma}l_{g\gamma}a_{g\gamma}\eta(s)u_{g\gamma},\text{
}m_{g\gamma}=l_{g\gamma}a_{g\gamma},
\]
as in \S 3, just before Theorem 3.2.

Now let $\nu^{\prime}:\mathbb{R}c\rightarrow\mathbb{C}$ be a second, real
linear map satisfying (3.11) and let $\tilde{\psi}$ be a second cusp form on
$L_{\theta}^{\prime}$ $($and set $\psi=\tilde{\psi}\circ\tilde{\omega});$
then
\[
\{\wedge^{H_{0}}E_{\varphi}(\nu),\wedge^{H_{0}}E_{\psi}(\nu^{\prime})\}
\]%
\[
=_{df}\int_{\hat{K}\backslash\hat{G}^{\lambda}/\hat{\Gamma}}\wedge^{H_{0}%
}E_{\varphi}(\nu)(g\eta(s))\overline{\wedge^{H_{0}}E_{\psi}(\nu^{\prime}%
)}(g\eta(s))dg,
\]
and one obtains that this last expression equals%
\[
=-\{\varphi,\psi\}_{L_{\theta_{4}}^{\prime}}\frac{\exp((\sigma+\bar{\sigma
}^{\prime})(H_{0}))}{(\sigma+\bar{\sigma}^{\prime})(c)}.
\]
The notation here is as follows:%
\[
\sigma=\nu+\rho,\text{ }\sigma^{\prime}=\nu^{\prime}+\rho,
\]
$\{,\}_{L_{\theta_{4}}^{\prime}}$ denotes the inner product induced from a
suitable Haar measure on $L_{\theta_{4}}^{\prime}/(\Gamma_{\theta_{4}}%
^{r})^{\prime},$ and $da$ is a suitable Haar measure on $A(M_{\theta_{4}}).$
To obtain this result, one uses the methods of [MS3], [MS4]. As in [MS3], one
first replaces the Eisenstein series $E_{\varphi}(\nu)$ (and similarly,
$E_{\psi}(\nu^{\prime}))$ by a pseudo-Eisenstein series: let $\Phi=\Phi(a)$ on
$A(M_{\theta_{4}})$ be a $C^{\infty}$ function with compact support, and let%
\[
E_{\varphi}(\Phi)(g\eta(s))=\sum_{\gamma\in\hat{\Gamma}/\hat{\Gamma}\cap
\hat{P}_{\theta_{4}}}\varphi(m_{g\gamma})\Phi(a_{g\gamma});
\]
then $E_{\varphi}(\Phi)$ is called a pseudo-Eisenstein series. One lets%
\[
\hat{\Phi}(\mu)=\int_{A(M_{\theta_{4}})}\Phi(a)\exp(-(\mu-\rho)(\log
a))d\mu_{I},
\]
where the notation is as follows: $\mu:\mathbb{R}c=\mathfrak{h}_{\theta_{4}%
}\rightarrow\mathbb{C}$ is real linear, and $\mu_{I}$ denotes the imaginary
part of $\mu.$ One can then define the truncation $\wedge^{H_{0}}E_{\varphi
}(\Phi),$ just as we defined $\wedge^{H_{0}}E_{\varphi}(\nu),$ and then for
$\Psi$ a second $C^{\infty}$ function with compact support on $A(M_{\theta
_{4}}),$we have for $\mu_{0},$ $\mu_{0}^{\prime}:\mathbb{R}c\rightarrow
\mathbb{R}$ with%
\[
\mu_{0}(c)<-g,\text{ }\mu_{0}^{\prime}(c)<-g,
\]
that, similarly to [MS3],%
\[
\{\wedge^{H_{0}}E_{\varphi}(\Phi),\text{ }\wedge^{H_{0}}E_{\psi}(\Psi)\}
\]%
\[
=-\int_{\operatorname{Re}\mu=\mu_{0}}\int_{\operatorname{Re}\mu^{\prime}%
=\mu_{0}^{\prime}}\hat{\Phi}(\mu)\overline{\hat{\Psi}(\mu^{\prime})}\Xi
(\mu,\bar{\mu}^{\prime})d\mu_{I}d\mu_{I}^{\prime},
\]
where%
\[
\Xi(\mu,\bar{\mu}^{\prime})=\{\varphi,\psi\}_{L_{\theta_{4}}^{^{\prime}}}%
\frac{\exp(\mu+\bar{\mu}^{\prime})(H_{0})}{(\mu+\bar{\mu}^{\prime})(H_{0})}.
\]
The argument in the present setting is in fact simpler than that in [MS3]: one
does not have to contend with the infinite sums over the affine Weyl group
that appear in [MS3], and one does not need to use the functional equation for
$c$-functions, in order to show that certain poles cancel, and so, as a
result, that one can move the contours of certain integrals past these
(non-existent) poles. In the present setting, there are no poles from the
$c$-functions, since non-trivial $c$-functions don't even occur in the formula
for $\Xi(\mu,\bar{\mu}^{\prime})!$

Finally, we can pass from the inner product for truncated pseudo-Eisenstein
series to that for truncated Eisenstein series, as in [MS4]. In particular, we
obtain that the truncated Eisenstein series $\wedge^{H_{0}}E_{\varphi}(\nu)$
is square summable.

Now the above computation is valid for%
\[
\operatorname{Re}\sigma(c)<-g,
\]%
\[
\operatorname{Re}\sigma^{\prime}(c)<-g,
\]
or equivalently%
\begin{equation}
\operatorname{Re}\nu(c)<-2g,\tag{4.2}%
\end{equation}%
\[
\operatorname{Re}\nu^{\prime}(c)<-2g.
\]
But clearly (in $\nu,$ $\nu^{\prime})$ the right side of the equality
(Maass-Selberg relation)%
\begin{equation}
\{\wedge^{H_{0}}E_{\varphi}(\nu),\wedge^{H_{0}}E_{\psi}(\nu^{\prime
})\}\tag{4.3}%
\end{equation}%
\[
=\{\varphi,\psi\}_{L_{\theta}^{\prime}}\frac{\exp((\sigma+\bar{\sigma}%
^{\prime})(H_{0}))}{(\sigma+\bar{\sigma}^{\prime})(c)}%
\]
is holomorphic in the region (4.2), and in fact, has a holomorphic extension
to the region%
\begin{equation}
\operatorname{Re}\nu(c)<-g,\tag{4.4}%
\end{equation}%
\[
\operatorname{Re}\nu^{\prime}(c)<-g;
\]
i.e.,%
\begin{equation}
\operatorname{Re}\sigma(c)<0,\text{ }\operatorname{Re}\sigma^{\prime
}(c)<0.\tag{4.4$^\prime$}%
\end{equation}
From this one can deduce that the Eisenstein series has a \textit{holomorphic}
continuation (in $\nu),$ as a locally integrable function, to the region%
\[
\operatorname{Re}\nu(c)<-\rho(c)=-g.
\]
We emphasize again: this is a \textit{holomorphic} continuation! -
\textit{not} just a meromorphic one.

We note that the validity of (4.3) only depends on our assumption that
$\hat{P}_{\theta}$ is not associate to any $\hat{P}_{\theta^{\prime}},$
$\theta^{\prime}\neq\theta,$ by virtue of $\mathfrak{l}_{\theta},$
$\mathfrak{l}_{\theta^{\prime}}$ not being isomorphic to one another
($\mathfrak{l}_{\theta},$ $\mathfrak{l}_{\theta^{\prime}}$ being the Lie
algebras of $L_{\theta},$ $L_{\theta^{\prime}},$ respectively). There are of
course many instances other than the case of $(\hat{E}_{6},\theta_{4})$
considered earlier, where this assumption holds; e.g.,%
\[
(\hat{E}_{7},\theta),\text{ }\theta=\Xi-\{\alpha_{4}\},
\]%
\[
\Xi=\text{ set of simple roots, }\alpha_{4}\text{ as in [Bourb].}%
\]

\section{Local Issues: A Summary of where Things Stand.}

The question remains: Are there applications of (4.3) and the holomorphic
continuation of loop Eisenstein series to the theory of $L$-functions, as in
the finite-dimensional case treated in [L]? The seeming paradox here is that
the argument for (4.3) (which is based on Shahidi's lemma 4.1) also seems to
preclude obtaining new results on $L$-functions: For $(\hat{E}_{6},\theta
_{4})$ for example, Lemma 4.2 implies that $L$-functions do not even occur in
the constant term.

It was Braverman and Kazhdan who pointed to a possible way out of this
dilemma: they argued that though the constant terms with respect to the "upper
triangular" $\hat{P}_{\theta}$ are trivial, one could consider the constant
terms with respect to "lower triangular" parabolics. By a "lower triangular"
parabolic one means a proper subgroup of $\hat{G}_{k}^{\lambda}$, $k$ a field,
containing the group of elements in $\hat{G}_{k}^{\lambda}$ which are lower
triangular with respect to the coherently ordered basis $\mathcal{B}$ (a
"lower triangular" Borel subgroup). An upper "upper triangular" parabolic is
simply a prabolic subgroup as defined in \S 2.

One expects that any computation of such constant terms would depend on local
computations, and in particular, would depend on suitable Gindikin-Karpelevich
formulae. These formulae would have to be established for the following three
cases: (i). $k=\mathbb{R}$ or $\mathbb{C},$ (ii) $k=F((t)),$ $F$ a finite
field, and (iii) $k=$ a finite algebraic extension $\mathcal{K}$ of a $p$-adic
completion of the rational numbers. We note that the results of \S \S 1-4,
above, can all be developed equally well for function fields over finite
fields.Concerning cases (ii) and (iii), a Gindikin-Karpelevich formula was
conjectured in [BFK]. This formula was derived by assuming that a certain
result in [BFG] for F((t)), F a field of characteristic 0, was also valid for
F a finite field. Recently, A. Braverman informed me that this was in fact the
case. The resulting formula for case (ii) then also suggested the formula for
case (iii). In [BGKP], we prove this conjecture (with a small modification)
for both cases (ii) and (iii). The proof in [BGKP] is based on a formula of A.
Braverman, D. Kazhdan, and M. Patnaik, for spherical functions on p-adic loop
groups and on loop groups over $F((t)),$ $F$ a finite field.

\section{ Bibliography}

[A] J. Arthur, A trace formula for reductive groups. II: applications of a
truncation operator, Compos. Math. \textbf{40}(1980), 87-121.

\bigskip

[Bourb] N. Bourbaki, Groupes et alg\`{e}bres de Lie, Chapitres 4, 5 et 6,
Hermann, Paris(1968).

\bigskip

[BFG] A. Braverman, M. Finkelberg and D. Gaitsgory, Uhlenbeck spaces via
affine Lie algebras, The unity of mathematics, 17-135, Progr. Math.,
\textbf{244}, Birkh\"{a}user, Boston MA, 2006.

\bigskip

[BFK] A. Braverman, M. Finkelberg, D. Kazhdan, "Affine Gindikin-Karpelevich
formula via Uhlenbeck spaces, arXiv: 0912.5132 v.2 [math.RT].

\bigskip

[BGKP] A. Braverman, H. Garland, D. Kazhdan and M. Patnaik, A Gindikin
Karpelevich formula for loop groups over non-archimedian, local fields, in preparation.

\bigskip

[GMRV] M. B. Green, S. D. Miller, J. G. Russo, P. Vanhove, Eisenstein series
for higher-rank groups and string theory amplitudes, arXiv: 1004.0163 v. 2 [hep-th].

\bigskip

[LA] H. Garland, The arithmetic theory of loop algebras, J. Algebra
\textbf{53}(1978), 480-551.

\bigskip

[LG] H. Garland, The arithmetic theory of loop groups, Inst. Hautes \'{E}tudes
Sci. Publ. Math. \textbf{52}(1980), 5-136.

\bigskip

[R] H. Garland, Certain Eisenstein series on loop groups: convergence and the
constant term, Proceedings of the International Conference on Algebraic Groups
and Arithmetic (in Honor of M. S. Raghunathan), December 2001 (S. G. Dani and
Gopal Prasad, eds.), Tata Institute of Fundamental Research, Mumbai, India,
2004, 275-319.

\bigskip

[AC] H. Garland, Absolute convergence of Eisenstein series on loop groups,
Duke Math. J. \textbf{135}(2006), 203-260.

\bigskip

[MS2] H. Garland, Eisenstein series on loop groups: Maass-Selberg relations 2,
Amer. J. Math. \textbf{129}(2007), 723-784.

\bigskip

[MS3] H. Garland, Eisenstein series on loop groups: Maass-Selberg relations 3,
Amer. J. Math. \textbf{129}(2007), 1277-1353.

\bigskip

[MS4] H. Garland, Eisenstein series on loop groups: Maass-Selberg relations 4,
Contemp. Math. \textbf{442}(2007), 115-158,  Proceedings of the conference
"Lie Algebras, Vertex Operator Algebras and their Applications" in honor of
James Lepowsky and Robert Wilson.

\bigskip

[L] R. P. Langlands, Euler products, Yale Mathematical Monographs \textbf{1},
Yale University Press, New Haven CT(1971).

\bigskip

[L2] R. P. Langlands, On the functional equations satisfied by Eisenstein
series, Lecture Notes in Mathematics \textbf{544}, Springer-Verlag, New York (1976)

\bigskip

[Lo] P. J. Lombardo, The constant terms of Eisenstein series on affine
Kac-Moody groups over function fields, Ph.D. thesis, University of
Connecticut, Storrs, CT (2010).

\bigskip

[S] F. Shahidi, Infinite dimensional groups and automorphic $L$ -functions,
Pure Appl. Math. Q. \textbf{1}(2005), 683-699.

\end{document}